\documentclass[12pt]{article}

\usepackage[a4paper,includeheadfoot,margin=2.54cm]{geometry}
\usepackage[latin2]{inputenc}
\usepackage[english]{babel}
\usepackage{amsthm}
\usepackage{amsmath}
\usepackage{amssymb}
\usepackage[hyphens]{url}

\theoremstyle{plain}
\newtheorem{thm}{Theorem}
\newtheorem{lemma}{Lemma}

\newtheorem{example}{Example}[section]
\theoremstyle{definition}

\theoremstyle{remark}
\newtheorem{remark}[thm]{Remark}

\newcommand{\fq}{\mathbb{F}_q}

\newcommand{\E}{\mathbf{E}}

\newcommand{\ecq}{\E(\fq)}

\newcommand{\ffq}{\fq(\E)}
\newcommand{\C}{\mathbf{C}}
\newcommand{\ffC}{\fq(\C)}

\begin{document}
\title{On the linear complexity profile of some sequences derived from elliptic curves}
\author{
L\'aszl\'o M\'erai, 
Arne Winterhof\\
{\small Johann Radon Institute for Computational and Applied Mathematics}\\
{ \small Austrian Academy of Sciences}\\
{ \small Altenbergerstr.\ 69, 4040 Linz, Austria}\\
{ \small e-mail: merai@cs.elte.hu, arne.winterhof@oeaw.ac.at}}
\maketitle

\begin{abstract}

For a given elliptic curve $\E$ over a finite field of odd characteristic and a rational function $f$ on $\E$ we first study the linear complexity profiles of the sequences 
$f(nG)$, $n=1,2,\dots$ which complements earlier results of Hess and Shparlinski. We use Edwards coordinates to be able to deal with many $f$ where Hess and Shparlinski's result does not apply.
Moreover, we study the linear complexities of the (generalized) elliptic curve power generators $f(e^nG)$, $n=1,2,\dots$
We present large families of functions $f$ such that the linear complexity profiles of these sequences are large.
\end{abstract}

\textit{2000 Mathematics Subject Classification: Primary} 65C10, 14H52, 94A55, 94A60

\textit{Key words and phases:} linear complexity, elliptic curves, Edwards coordinates, elliptic curve generator, power generator, elliptic curve power generator

\let\thefootnote\relax\footnote{The final publication is available at Springer via \url{http://dx.doi.org/10.1007/s10623-015-0140-0}}

\section{Introduction}

The \textit{linear complexity profile} $L(s_n, N)$, $N=1,2,\dots$, of a sequence $(s_n)$ over a ring $R$ is a non-decreasing sequence where the $N$-th term is defined as the length $L$ of a shortest linear recurrence relation
\begin{equation*}
 s_{n+L}=c_{L-1}s_{n+L-1}+\dots +c_1s_{n+1}+c_0s_n, \quad 0\leq n\leq N-L-1
\end{equation*}
for some $c_0,\dots, c_{L-1}\in R$, that $(s_n)$  satisfies,  
with the convention that $L(s_n,N)=0$ if the first $N$ elements of $(s_n)$ are all zeros, and $L(s_n,N)=N$ if $s_0=\dots=s_{N-2}=0$ and $s_{N-1}\neq 0$. The value
\[
 L(s_n)=\sup_{N\geq 1} L(s_n,N)
\]
is called the \textit{linear complexity} of the sequence $(s_n)$.

The linear complexity profile measures the unpredictability of a sequence and thus its suitability in cryptography. 
For more details see \cite{MeidlWinterhof, Niederreiter2003, Winterhof2010}.

A common method for generating pseudorandom sequences is the \textit{linear method}. Namely, for integers $a,b,m$ and $\vartheta$ with $\gcd(a,m)= \gcd(\vartheta ,m)=1$  we define the sequence $(x_n)$ as
\begin{equation}\label{eq:linear}
 x_{n}\equiv ax_{n-1}+b \pmod m ,\quad 0\leq x_n<m, \quad n=1,2,\dots
\end{equation}
with the initial value $x_0=\vartheta$, see \cite{knuth, NW}. Although linear generators have many applications including Monte-Carlo integration, they have linear complexity profile $L(x_n,N)\leq 2$; so they are highly predictable and thus unsuitable in cryptography.

A more adequate method for cryptographic applications is the \textit{power generator}. Namely, let $\vartheta$, $m$ and $e$ be integers such that $\gcd (\vartheta,m)=1$. Then one can define the sequence $(u_n)$ by the recurrence relation
\begin{equation}\label{eq:power}
 u_n \equiv u_{n-1}^{e} \pmod m, \quad 0\leq u_n <m, \quad n=1,2,\dots
\end{equation}
with the initial value $u_0=\vartheta$.

Two special cases of the power generator (both for $m=pq$ the product of two primes) are the \textit{RSA generator}, when $\gcd (e,\varphi(m))=1$, where $\varphi$ is the Euler function, and the \textit{Blum-Blum-Shub generator} (or \textit{square generator}), when $e=2$. The linear complexities of these generator were studied by Griffin and Shparlinski \cite{GriffinShparlinski}, and Shparlinski \cite{shparlinski2001}.

For more background on pseudorandom number generation we refer to the survey articles \cite{TopWinterhof,Winterhof-SETA2010} and the monographs \cite{N,NW}.

In this paper we study the linear complexities of the elliptic curve analogues of the sequences defined by \eqref{eq:linear} and \eqref{eq:power}.

In Section~\ref{sec:EC} we summarize some basic facts about elliptic curves. In Section~\ref{sec:main}, 
we define the elliptic curve generators and elliptic curve power generators with respect to a rational function of the curve. 
Next, by using Edwards coordinates we state a complement to a result of Hess and Shparlinski \cite{HessShparlinski} for a large family of functions where \cite[Theorem~4]{HessShparlinski} is not applicable.
Then we present an extension of a result of Lange and Shparlinski \cite{LangeShparlinski} on the linear complexity  of the elliptic curve power generator defined via the first coordinate
to analogues defined via more general rational functions. Finally,  in Sections~\ref{sec:proof-con} and \ref{sec:proof-power} we present the proofs.

We emphasize the two new ideas in the proofs compared to \cite{HessShparlinski,LangeShparlinski}. First, the method of \cite{HessShparlinski} fails if a certain pole divisor on the elliptic curve is not of a very special form, see Section \ref{sec:con} for more details. However, if we use Edwards coordinates some different but rather mild conditions have to be satisfied. Consequently, we can deal with many more functions not covered by \cite{HessShparlinski}. Secondly, a more general linear independence property than in \cite{LangeShparlinski} from \cite{merai-ec-power} is used to extend the results of \cite{LangeShparlinski}.

We use the notation $A(x)\ll B(x)$ or $B(x)\gg A(x)$ if $|A(x)|\leq c B(x)$ holds for some positive constant $c$.

\section{Elliptic curves}\label{sec:EC}

Let $\fq$ be the finite field of $q$ elements with a prime power $q$,
and let $\E$ be an elliptic curve defined by the \textit{Weierstrass equation}
\[
 y^2+(a_1x +a_3)y=x^3+a_2x^2+a_4x+a_6
\]
with $a_1,a_2,a_3,a_4,a_6\in \fq$ and non-zero discriminant (see \cite{washington}).

The $\fq$-rational points $\ecq$ of $\E$ form an Abelian group (with respect to the usual 'geometric' addition which we denote by $\oplus$) with the point at infinity $\mathcal{O}$ as the neutral element. We also recall that
\[
| \#  \ecq -q-1|\leq 2 q^{1/2}, 
\]
where $\#  \ecq$ is the cardinality of $\ecq$.

For a positive integer $m$ let $\E[m]$ be the set of \emph{$m$-torsion points}:
\[
 \E[m]=\{P\in\E(\overline{\mathbb{F}_q}): \ mP= \mathcal{O}\}, 
\]
where $\overline{\mathbb{F}_q}$ is the algebraic closure of $\mathbb{F}_q$.
It is well-known, see for example \cite[Theorem~3.2]{washington} that for $m$ with $\gcd(m,q)=1$ we have
\[
 \E[m]\cong \mathbb{Z}_m \times \mathbb{Z}_m.
\]
On the other hand, if  $m=p^{\nu}m'$ with $p \nmid m'$, where $p$ is the characteristic of $\fq$, then either
\[
  \E[m]\cong \mathbb{Z}_{m'} \times \mathbb{Z}_{m'} \quad  \text{or}  \quad \E[m]\cong \mathbb{Z}_{m} \times \mathbb{Z}_{m'}.
\]

In 2007, Edwards introduced an alternative representation of elliptic curves called \textit{Edwards curves} \cite{Edwards}  (see also \cite{BrensteinLange, washington}). For a finite field $\fq$ of odd characteristic, an Edwards curve $\C$ is defined by the equation
\begin{equation*}
u^2+v^2=c^2(1+du^2v^2),
\end{equation*}
where $c,d\in\fq$, $d\neq 0,1$, $c\neq 0$. For a non-square $d$ over $\fq$ the addition is defined by
\begin{equation}\label{eq:addition-twisted}
(u_1,v_1)\oplus (u_2, v_2)=\left(\frac{u_1v_2+u_2v_1}{c(1+du_1u_2v_1v_2)},  \frac{v_1v_2-u_1u_2}{c(1-du_1u_2v_1v_2)} \right).
\end{equation}
The points of the curve form a group with respect to this addition, with $(0,c)$ as the neutral element. 
We remark, that every Edwards curve is birationally equivalent to an elliptic curve. On the other hand, if $\ecq$ has points of order four and a unique point of order two, 
then $\E$ is birationally equivalent to an Edwards curve with $c=1$ (see Theorem 2.1 in \cite{BrensteinLange}). Namely, $\C$, with $c=1$, is isomorphic to the elliptic curve $\E$
defined by
\[
 y^2=x^3+2(1+d)x^2+(1-d)^2x
\]
where the isomorphism is given by
\begin{equation}\label{iso}
 \begin{array}{ccl}
  \psi: \ \ecq \setminus \{\mathcal{O},(0,0)\} & \rightarrow & \C \setminus \{(0,1),(0,-1)\}\\
  (x,y) & \mapsto & \left\{ \begin{array}{l}
                             u=2\frac{x}{y}\\
                             v=\frac{x-1+d}{x+1-d}.
                            \end{array}
  \right.
 \end{array}
\end{equation}
(Note that the other two points in $\E (\overline{\mathbb{F}_q})$ with $y=0$ are not in $\ecq$.)
We can extend $\psi$ by setting $\psi (\mathcal{O})=(0,1)$ and $\psi((0,0))=(0,-1)$.

All the $\fq$-rational points of the Edwards curve are affine, but there are two ideal points (points at infinity) $\Omega_1$ and $\Omega _2$ over the quadratic extension $\mathbb{F}_{q^2}$. 
More precisely, let us consider the embedding of $\C$ into the projective plane
\[
 U^2Z^2+V^2Z^2=c(Z^4+dU^2V^2).
\]
The affine points $(u,v)$ correspond to $(u:v:1)$, $u,v\in \fq$, and the ideal points are $\Omega_1=(1:0:0)$ and $\Omega_2=(0:1:0)$. 
The addition $(U_1:V_1:Z_1)\oplus(U_2:V_2:Z_2)=(U_3:V_3:Z_3)$ with projective coordinates is defined by
\begin{equation}\label{eq:addition-projective}
\begin{split}
 A=Z_1Z_2; \ B=A^2; \ C=U_1\cdot U_2; \ D=V_1\cdot V_2; \ E=d\cdot C\cdot D; F=B-E;\ G=B+E;\\
 U_3=A\cdot F\cdot ((U_1+V_1)\cdot (U_2+V_2)-C-D); \ V_3=A\cdot G\cdot (D-C); \ Z_3=c\cdot F\cdot G.
\end{split}
\end{equation}

We remark that the addition laws \eqref{eq:addition-twisted} and \eqref{eq:addition-projective} are not complete over the quadratic extension  $\mathbb{F}_{q^2}$ but they can be extended to sets of two addition laws which allows the addition on Edwards curve over arbitrary field extensions (see \cite{BrensteinLange2}).

\section{Main results}\label{sec:main}

In this section we present results on the linear complexity of the elliptic curve analogues of the linear generator \eqref{eq:linear} and the power generator \eqref{eq:power}. These results will be proven in 
Sections~\ref{sec:proof-con} and \ref{sec:proof-power}.

\subsection{Elliptic curve generator}\label{sec:con}

For $f\in\ffq$  the \textit{elliptic curve  generator  $(w_n)$ with respect to $f$}  is the sequence
\[
 w_n=f(nG), \quad n=1,2,\dots,
\]
with $G\in\ecq$.

The linear complexity profile of the sequence $(w_n)$ has already been studied for special functions $f$, \cite{HessShparlinski,TopWinterhof}. In particular, Hess and Shparlinski \cite{HessShparlinski} proved that if the pole divisor $(f)_\infty$ of $f$ is of the form
\[
(f)_\infty=(1+\delta)H
\]
for some place $H$ of the curve and
\[
 \delta=\left\{
 \begin{array}{cl}
  1 & \text{if } \deg H =1,\\
  0 & \text{if } \deg H \geq 2,\\
 \end{array}
 \right.
\]
then
\[
 L(w_n,N)\geq \min\left\{\frac{N}{(1+\delta)\deg H +2}, \frac{t}{(1+\delta)\deg H +1} \right\}
\]
where $t$ is the order of $G$.

By using Edwards coordinates, we can give another large family of functions $f$ such that the linear complexity profiles of the corresponding sequences are large.

\begin{thm}\label{thm1}
Let $\C$ be an Edwards curve and $f\in\ffC$ such that $\Omega_1$ or $\Omega_2$ is a pole of~$f$. 
If  $G\in \C$ of order $t$ and $w_n=f(nG)$, then 
\[
L(w_n, N)\geq \min \left\{ \frac{t-\deg f}{4\deg f}, \frac{N-\deg f}{4\deg f +1} \right\}, \quad N\geq \deg f.
\]
\end{thm}

We remark that $\deg f$ is the degree of the pole divisor of $f$, especially, $\deg u=\deg v=2$, where $u$ and $v$ are the coordinate functions.

\begin{example}
Let $f\in\ffC$ (with $c=1$) be the sum of the coordinate functions: $f(u,v)=u+v$. Then Theorem~\ref{thm1} implies that the linear complexity profile of the sequence $w_n=(u+v)(nG)$ satisfies
\[
 L(w_n, N)\geq \min \left\{ \frac{t-4}{16}, \frac{N-4}{17} \right\}, \quad N\geq 4,
\]
since both $\Omega_1$ and $\Omega_2$ are poles. On the other hand, transforming the coordinates into an elliptic curve we get by $(\ref{iso})$
\[
 (u+v)(nG)=\left(2\frac{x}{y} + \frac{x-1+d}{x+1-d}\right) (n\overline{G}),
\]
where $\overline{G}=\psi^{-1}(G)$ is the isomorphic image of $G$ on $\E(\fq)$. Since the pole divisor of
\[
 g(x,y)=2\frac{x}{y} + \frac{x-1+d}{x+1-d}\in\ffq
\]
is
\[
 (g)_\infty=\left( x^2+2(1+d)x+(1-d)^2\right)_0 +\left( x+1-d\right)_0, 
\]
the Hess-Shparlinski bound cannot be applied. (Here $(h)_0$ is the zero divisor of $h\in\E(\fq)$.)
\end{example}

\subsection{Elliptic curve power generator}
For a positive integer $e>1$ and a point $G\in\ecq$ of order $|G|$ with $\gcd(e, |G|)=1$, consider the elliptic curve analogue of \eqref{eq:power} defined by
\begin{equation}\label{eq:power-ec}
 G_n=eG_{n-1}=e^nG_0, \quad n=1,2,\dots
\end{equation}
(with $G_0=G$). Determining $e$ from a pair $(G_n,G_{n-1})$ would solve the discrete logarithm problem on $\E$, while  computing $G_n$ from previous elements (without knowing $e$) is related to the elliptic curve Diffie-Hellman problem, thus the generator  \eqref{eq:power-ec} is thought to be `secure'.

Clearly, the sequence $(G_n)$ is periodic, and the period length is the multiplicative order $t$ of $e$ modulo $|G|$.

In this section we study the sequences obtained from the coordinates of \eqref{eq:power-ec}. Namely, 
for an $f\in\ffq$  the \textit{elliptic curve  power generator  $(r_n)$ with respect to $f$}  is the sequence
\[
 r_n=f(G_n)=f(e^nG), \quad n=1,2,\dots
\]

The linear complexity of the sequence $(r_n)$ for the coordinate function $f(x,y)=x$ was studied by Lange and Shparlinski \cite{LangeShparlinski}. They proved, that if $\E$ is non-supersingular, then
\begin{equation*}%\label{eq:LangeShparlinski}
 L(r_n)\gg t |G|^{-2/3}.
\end{equation*}
 
We can extend their result.
 
\begin{thm}\label{thm:power}
Let $\fq$ be a finite field, let $\E$ be an elliptic curve over $\fq$ and let $f\in\fq [\E]$ be a non-constant function of degree $\deg f <|G|^{\delta}$ for some $\delta <1$. 
If the  multiplicative order of $e$ modulo $|G|$ is $t$, then
\[
 L(r_n)\gg  \frac{t}{|G|^{2/3} (\deg f)^{1/3}},
\]
where the implied constant depends on $\delta$.
\end{thm}

\section{Proof of Theorem~\ref{thm1}}\label{sec:proof-con}

Theorem~\ref{thm1} is based on the following lemma.

\begin{lemma}\label{lemma:not-constant}
Let $f\in\ffC$ be a rational function such that $\Omega_1$ or $\Omega_2$ is a pole of $f$ and $G$ a point on $\C$ of order $t$.  Then, for any integer $L$ with $1\leq L<t/8$ and any coefficients $c_0,\dots, c_L\in\fq$ with $c_L\neq 0$ we have that
\begin{equation}\label{eq:F-def}
  F(Q)=\sum_{l=0}^Lc_lf(Q\oplus lG)\in\ffC
\end{equation}
is not constant and has degree 
\[
\deg F\leq (4L+1)\deg f.
\]
\end{lemma}

\begin{proof}
We may assume $t> 8$. 
Write $(u_n,v_n)=nG$. First we show, that for $1\leq j<L< t/8$ we have $u^2_Lv^2_L\neq u^2_jv^2_j$. Indeed, consider the function
\[
 H(u,v)=u^2v^2-u^2_Lv^2_L.
\]
It has at most $\deg H =8$ zeros. On the other hand, the set of zeros is closed under the transformations
\begin{equation}\label{eq:trasformations}
 (u,v)\mapsto (v,u), \quad  (u,v)\mapsto (-u,v).
\end{equation}
Note that $(v,u), (-u,v)\in \C(\fq)$.
If there were a $j$ with $1\leq j<L< t/8$ and $H(u_j,v_j)=0$, then the orbit of $\{(u_j,v_j), (u_L,v_L)\}$ under the transformations \eqref{eq:trasformations} would contain $16$ zeros of $H$ 
(since $1\leq j<L< t/8$), a contradiction.

By \eqref{eq:addition-twisted} we have in $\C(\fq)$
\begin{equation}\label{eq:F}
  F(u,v)=\sum_{l=0}^Lc_lf\left(\frac{uv_l+u_lv}{c(1+duu_lvv_l)},  \frac{vv_l-uu_l}{c(1-duu_lvv_l)} \right).
\end{equation}
Define $P_1,P_2\in\mathbb{F}_{q^2}(\C)$ by
\[
P_1=\left(\frac{1}{\sqrt{d}u_L} , \frac{-1}{\sqrt{d}v_L}\right)\quad \text{and} \quad P_2=\left(\frac{1}{\sqrt{d}v_L} , \frac{1}{\sqrt{d}u_L}\right).
\]
Then we have
\[
 (u_L,v_L)\oplus P_1=\Omega_1 \quad \text{and} \quad  (u_L,v_L)\oplus P_2=\Omega_2,
\]
by $(\ref{eq:addition-projective})$,
but all the points $(u_n,v_n)\oplus P_i$ are affine for $0\leq n <L$ and $i=1,2$. Thus if $\Omega_i$ is a pole of $f$, then $P_i$ is a pole of the $L$-th 
term of the right hand side of $\eqref{eq:F}$, but not a pole of any other term, so $P_i$ is a pole of $F$. Hence, $F$ is not constant.

For $P=(u_0,v_0)\neq (0,c)$ the function $f_P: Q \mapsto f(Q \oplus P)$ has degree at most $4\deg f$, namely, if $R$ is a pole of $f$, then $R\oplus (-(x_0,y_0))$ and $R\oplus (-(y_0,x_0))$ are poles of $f_P$  and their multiplicities are at most twice of the multiplicity of $R$. Thus
\[
 \deg \left(\sum_{l=0}^{L}c_l f_{lG}\right)\leq \deg f +L\cdot 4\deg f.
\]

\end{proof}

\begin{proof}[Proof of Theorem~\ref{thm1}]
We may assume that $L<t/8$, since otherwise the theorem trivially holds.

Put  $c_L=-1$ and assume that
\begin{equation*}
\sum_{l=0}^L c_l w_{n+l}=0, \quad 0\leq n\leq N-L-1.
\end{equation*}
Whence
\begin{equation*}
\sum_{l=0}^L c_l f((n+l)G)=0, \quad 0\leq n\leq N-L-1,
\end{equation*}
so the function $F$ defined in \eqref{eq:F-def} has at least $\min\{N-L,t\}$ zeros, namely the points $nG$ with $0\leq n \leq \min\{N-L,t\}-1$. 
On the other hand the degree of $F$ is at most $(4L+1) \deg f$, thus the result follows from Lemma~\ref{lemma:not-constant}.
\end{proof}

\section{Proof of Theorem~\ref{thm:power}}\label{sec:proof-power}

We need the following basic lemma about linear complexity (\cite[Lemma~2]{shparlinski2001}).

\begin{lemma}\label{lemma:shp}
Let a sequence $(s_n)$ satisfy a linear recurrence relation 
\[
s_{n+L}=a_{L-1}s_{n+L-1}+\dots +a_1s_{n+1}+a_0s_n, \quad n=1,2,\dots
\] 
over $\fq$. Then for any $T\geq L+1$ pairwise distinct non-negative integers $j_1,\dots, j_T$ there exist $c_1,\dots, c_T \in \fq$, not all equal to zero, such that
 \[
  \sum_{i=1}^Tc_is_{n+j_i}=0, \quad n=1,2,\dots
 \]
\end{lemma}

We also need the following auxiliary result (Lemma 2 in \cite{FHS}).

\begin{lemma}\label{lemma:representation}
 Let $m>1$ be an integer. Then for any $\mathcal{K}\subset \mathbb{Z}_m^*$ of cardinality $|\mathcal{K}|=K$, any fixed $\delta>0$ and any integer $h\geq m^{\delta}$, there exists an integer $a\in  \mathbb{Z}_m^*$ such that the congruence
 \[
  k\equiv as \mod m, \quad k\in \mathcal{K}, \ 0\leq s \leq h-1,
 \]
has
\[
 T_a(h)\gg \frac{K h}{m}
\]
solutions $(k,s)$.
\end{lemma}

\begin{proof}[Proof of Theorem~\ref{thm:power}]
Put
\[
 \mathcal{K}=\{e^j: \ 0\leq j < t\}.
\]
Then by Lemma~\ref{lemma:representation} there is an $a$ and pairs $(j_1,s_1),\dots, (j_T,s_T)$ such that
 \[
  e^{j_i}\equiv as_{i} \mod |G|, \quad 0\leq j_i<t  , \ 0\leq s_i <   \left(\frac{|G|}{\deg f}\right)^{1/3},
 \]
for $i=1,\dots, T$, with 
\[
 T\gg\frac{t}{|G|^{2/3}(\deg f)^{1/3}}.
\]

If $L\geq T$, the theorem follows.
Now assume, that $L<T$. Then by Lemma~\ref{lemma:shp}
\[
\sum_{i=1}^{L+1}c_if\left(e^{n+j_i}G \right)=0, \quad n=1,2,\dots
\]
Now
\[
 e^{n+j_i}G =e^{j_i}e^nG=as_ie^nG=as_iG_n
\]
for all $n$. Thus the function
$$H(Q) = \sum_{i=1}^{L+1}c_if\left((as_i)Q \right)
$$
has at least $t \cdot \#\E[a]$ zeros, namely, all points of the form 
\[
G_n\oplus P,  \quad n=1,\dots,t, \quad  P\in \E[a],
\]
and it is not constant as it was proved in \cite[Theorem 2]{merai-ec-power}.

Clearly, the poles of $f_s:Q\mapsto f(asQ)\in\ffq$ are of the form
\[
 Q=R \oplus P, \quad \text{where $Q$ is a pole of $f$, and $P\in\E[as]$},
\]
thus
\[
 \deg f_{s}\leq  \deg f \cdot  \#\E[as]  \leq \deg f \cdot s^2\#\E[a].
\]
So the degree of $H$ is at most
\[
(L+1)\cdot \deg f\cdot \left(\max_{i=1,\dots, T} s_i \right)^2\cdot \#\E[a] \ll L \cdot \deg f \cdot\left(\frac{|G|}{\deg f}\right)^{2/3}\cdot \#\E[a]. 
\]
Comparing the number of zeros and the degree of $H$ 
we have
\[
 t  \ll L \cdot \deg f \cdot\left(\frac{|G|}{\deg f}\right)^{2/3} 
\]
which proves the theorem.
\end{proof}

\begin{remark}
 This proof is based on a linear independence property for any non-constant $f$ from the proof of \cite[Theorem 2]{merai-ec-power}, whereas in \cite{LangeShparlinski} only the special case is considered that $f$ is a linear combination of coordinate functions.
\end{remark}

\subsection*{Acknowledgment}
The authors are partially supported by the Austrian Science Fund FWF Project F5511-N26 
which is part of the Special Research Program "Quasi-Monte Carlo Methods: Theory and Applications". The first author is also partially supported by Hungarian National Foundation for Scientific Research,
Grant No. K100291.

\end{document}